\documentclass[9pt]{amsart}
\usepackage{amssymb,amsmath,geometry,latexsym,graphics,tabularx,shapepar,enumerate}

\newtheorem{Theorem}{Theorem}
\newtheorem{Lemma}[Theorem]{Lemma}
\newtheorem{Proposition}[Theorem]{Proposition}

\newtheorem{Definition}[Theorem]{Definition}
\newtheorem{Conjecture}[Theorem]{Conjecture}
\newtheorem{Remark}[Theorem]{Remark}

\newcommand{\AAA}{\boldsymbol{A}}

\newcommand{\irr}{I}
\newcommand{\irrr}{\star}
\newcommand{\irrbis}{I\!I}

\newcommand{\GL}{\operatorname{GL}}
\newcommand{\SL}{\operatorname{SL}}

\newcommand{\NN}{\mathbb N}

\newcommand{\ZZ}{\mathbb Z}

\newcommand{\CC}{\mathbb C}

\newcommand{\FF}{\mathbb F}

\newcommand{\sqm}[4]{\left(\begin{smallmatrix} #1 & #2 \\ #3 & #4 \end{smallmatrix}\right)}
\newcommand{\bigsqm}[4]{\left(\begin{matrix} #1 & #2 \\ #3 & #4 \end{matrix}\right)}

\newcommand\CVD{{\hfill\hfil{\lower 2 pt\hbox{\vrule\vbox to 7pt 
{\hrule width 6pt\vfill\hrule}\vrule}}}\vskip 0.5cm}

\date{\today\\ \footnotesize{This project has received funding from the European Research Council (ERC)
under the European Union's Horizon 2020 research and innovation programme under
the Grant Agreement No 648132.}}
\title{A note on certain representations in characteristic $p$
and associated functions}
\author{F. Pellarin}

\begin{document}

\begin{abstract}
The aim of this note is to describe basic properties of the representations of $\operatorname{GL}_2(\FF_q[\theta])$
associated to certain vectorial modular forms with values in Tate algebras and Banach algebras
introduced by the author. We discuss how certain $L$-values occur as limits values of these 
functions. We also present 
families of examples which can be the object of further studies.
\end{abstract}

\maketitle

\section{Introduction}

In \cite{Pe}, the author has introduced some special functions related to the arithmetic 
of function fields of positive characteristic (and more precisely, to the arithmetic of the 
ring $\FF_q[\theta]$ with $\theta$ an indeterminate), namely, $L$-values and vector valued 
modular forms (the vectors having entries in certain ultrametric Banach algebras). 
The purpose of \cite{Pe} was to produce a new class of 
functional identities for $L$-values, and only very particular examples of 
these new special functions were required, in order to obtain the results in that paper. The theory was later developed along 
several axes (see for example \cite{APTR}, which also contains quite a detailed bibliography).

The aim of this note is to highlight the connection that these special functions have
with representations of algebras, groups etc. associated to $A$, and to present 
families of examples which can be the object of further studies. In particular,
we are interested in certain irreducible representations of $\SL_2(\FF_q[\theta])$ 
or $\GL_2(\FF_q[\theta])$. We also provide a few explicit examples and properties of such representations.

The plan of the note is the following. In \S \ref{algebrarepr}, we discuss  
algebra representations of $A$ and we will consider their associated $\omega$-values and 
$L$-values. 
In \S \ref{symmetricpowers}, we present a class of irreducible representations $\rho^I$ inside  
symmetric powers in the case $q=p$. In \S \ref{tensorprod}, we apply the results of \S \ref{symmetricpowers} to show that certain tensor products $\rho^{I\!I}$ are irreducible. In 
\S \ref{poincare} we use these results to show that the entries of certain 
vectorial Poincar\'e series generalizing those introduced in \cite{Pe} are linearly independent
and we present a conjecture on the rank of a certain module of vectorial modular forms.

In all the following, $q=p^e$ with $p$ a prime number
and $e>0$. we set $$\Gamma=\operatorname{GL}_2(\FF_q[\theta]),$$ where 
$q=p^e$ for some prime number $p$ and an integer $e>0$. 
We shall write $A=\FF_q[\theta]$ (so that $\Gamma=\operatorname{GL}_2(A)$). 
All along the paper, if $a=a_0+a_1\theta+\cdots+a_r\theta^r$ is an element of $A$ with $a_0,\ldots,a_r\in\FF_q$
and if $t$ is an element of an $\FF_q$-algebra $B$, then $a(t)$ denotes the element $a_0+a_1t+\cdots+a_rt^r\in B$.
Also, we set $K=\FF_q(\theta)$. 

\section{Algebra representations}\label{algebrarepr}

In this section, 
we consider an integral, commutative $\FF_q$-algebra $\AAA$ and we denote by $\boldsymbol{K}$
its fraction field.
We denote by $\operatorname{Mat}_{n\times m}(R)$, with $R$ a commutative ring, the $R$-module
of the matrices with $n$ rows and $m$ columns, and with entries in $R$. If $n=m$, this $R$-module is equipped with the structure of an $R$-algebra.
 We choose an injective algebra representation
\begin{equation}\label{algrep}
A\xrightarrow{\sigma}\operatorname{Mat}_{d\times d}(\boldsymbol{K}),\end{equation}
which is completely determined by the choice of the image $\vartheta:=\sigma(\theta)$.
Note that $\sigma$ is not injective if and only if $\vartheta$ has all its eigenvalues in $\FF_q^{ac}$, 
algebraic closure of $\FF_q$ (In all the following, if $L$ is a field, $L^{ac}$ denotes an algebraic closure of $L$). Further, we have that $\sigma$ is irreducible if and only if its characteristic polynomial 
is irreducible over $\boldsymbol{K}$.

\subsection{An example}
We denote by $\AAA[\theta]^+$ the multiplicative monoid of polynomials which are monic in $\theta$. Let $P$ be a polynomial in 
$\AAA[\theta]^+$, let $d$ be the degree of $P$ in $\theta$.
The euclidean division in $\AAA[\theta]$ by $P$ defines for all $a\in \AAA[\theta]$, in an unique 
way, a matrix $\sigma_P(a)\in\operatorname{Mat}_{d\times d}(\AAA)$ such that
$$aw\equiv\sigma_P(a)w\pmod{P\AAA[\theta]},$$ where 
$w$ is the column vector with entries $1,\theta,\ldots,\theta^{d-1}$.
Explicitly, if $P=\theta^d+P_{d-1}\theta^{d-1}+\cdots+P_0$ with $P_i\in\AAA$, then 
$$\sigma_P(\theta)=\left(\begin{array}{cccc} 0 & 1 & \cdots & 0
\\ 
0 & 0 & \cdots & 0\\ \vdots & \vdots &  & \vdots \\ 
0 & 0 & \cdots & 1\\ 
-P_0 & -P_1 & \cdots & -P_{d-1}\end{array}\right).$$
Hence, the map $\sigma_P$ defines an algebra 
representation $$A\xrightarrow{\sigma_P}\operatorname{End}(\AAA^d).$$
The representation $\sigma_P$ is faithful if $P$ has not all its roots in $\FF_q^{ac}$ and 
is irreducible if and only if $P$ is irreducible over $\boldsymbol{K}$.

\subsection{$L$-values and $\omega$-values of algebra representations and semi-characters}

We give a few elementary properties of certain basic objects that can be associated to 
representations such as in (\ref{algrep}). Since the proofs are in fact obvious generalizations of the arguments of
\cite{Pe,APTR}, we will only sketch them.

For a ring $R$, we denote by 
$R^*$ the underlying multiplicative monoid of $R$ (if we forget the addition of $R$ we are left with 
the monoid $R^*$). We recall that $A^+$ denotes the multiplicative monoid of
monic polynomials of $A$.
Let $\boldsymbol{M}$ be an $\FF_q$-algebra (for example, $\boldsymbol{M}=\operatorname{Mat}_{d\times d}(\boldsymbol{K})$ for some integer $d$).
\begin{Definition}{\em 
A monoid homomorphism $$\sigma:A^+\rightarrow \boldsymbol{M}^*$$ is a {\em semi-character}
if there exist pairwise commuting $\FF_q$-algebra homomorphisms $$\sigma_1,\ldots,\sigma_s:A\rightarrow \boldsymbol{M}$$ such that, for $a\in A^+$, $$\sigma(a)=\sigma_1(a)\cdots\sigma_s(a).$$
The trivial map $\sigma(a)=1_{\boldsymbol{M}}$ for all $a$ is a semi-character, according to the convention that an 
empty product is equal to one.
If, for all $i=1,\ldots,s$, $\vartheta_i=\sigma_i(\theta)$ 
has a well defined minimal polynomial over $\boldsymbol{K}$, we say that the semi-character
$\sigma$ is of {\em Dirichlet type}. This happens if, for example, $\boldsymbol{M}=\operatorname{Mat}_{d\times d}(\boldsymbol{K})$.
The {\em conductor} of a semi-character of Dirichlet type is the product
of all the pairwise distinct minimal polynomials of the elements $\vartheta_1,\ldots,\vartheta_s$.}
\end{Definition}
\subsubsection*{Example} If we choose $\boldsymbol{M}=\FF_{q}^{ac}$, then a semi-character
$\sigma:A^+\rightarrow\FF_q^{ac}$
is always of Dirichlet type, and our definition coincides in fact with the usual notion of a Dirichlet-Goss 
character $A^+\rightarrow\FF_q^{ac}$. There are non-pairwise conjugated elements $\zeta_1,\ldots,\zeta_s\in\FF_q^{ac}$, of minimal polynomials $P_1,\ldots,P_s\in A$, such that 
$\sigma(a)=a(\zeta_1)^{n_1}\cdots a(\zeta_s)^{n_s}$ for all $a\in A$, with $0<n_i<q^{d_i}-1$ for all $i$,
with $d_i$ the degree of $P_i$. The conductor is the product $P_1\cdots P_s$.

\subsubsection*{Non-example} We set $\boldsymbol{M}=\FF_{q}[x]$ and we consider
the map $\sigma:A^+\rightarrow\boldsymbol{M}$ defined by $a\mapsto x^{\deg_\theta(a)}$.
Then, $\sigma$ is a monoid homomorphism which is not a semi-character. Indeed,
assuming the converse, then $\sigma=\sigma_1\cdots\sigma_s$ for 
algebra homomorphisms $\sigma_i:A\rightarrow\boldsymbol{M}$. But since $\sigma(\theta)=\sigma_1(\theta)\cdots\sigma_s(\theta)=x$, we get $s=1$ and $\sigma$ would be an algebra homomorphism, which is 
certainly false.

\medskip

From now on, we suppose, for commodity, that $\boldsymbol{M}=\operatorname{Mat}_{d\times d}(\boldsymbol{K}_s)$ with $\boldsymbol{K}_s=\FF_q(t_1,\ldots,t_s)$ (but some of the arguments also hold with $\boldsymbol{M}$ any $\FF_q$-algebra). Let $K_\infty$ be the completion of $K=\FF_q(\theta)$ at the infinity place. Then, $K_\infty=\FF_q((\theta^{-1}))$, with the norm $|\cdot|$ defined by $|\theta|=q$
(associated to the valuation $v_\infty$ such that $v_\infty(\theta)=-1$). 
Let $\CC_\infty$ be the completion $\widehat{K_\infty^{ac}}$, where $K_\infty^{ac}$ denotes
an algebraic closure of $K_\infty$.  We denote by $\mathbb{K}_s$ the completion 
of the field $\CC_\infty(t_1,\ldots,t_s)$ for the Gauss valuation extending the valuation of $K_\infty$, so that the valuation induced on $\FF_q^{ac}(t_1,\ldots,t_s)\subset\CC_\infty(t_1,\ldots,t_s)$
is the trivial one. Also, we denote by $K_{s,\infty}$ the completion of $K(t_1,\ldots,t_s)$
in $\mathbb{K}_s$; we have $K_{s,\infty}=\boldsymbol{K}_s((\theta^{-1}))$. We have that $K_\infty=K_{0,\infty}$, and $\mathbb{K}_0=\CC_\infty$.

\subsubsection{$\omega$-values of an algebra representation} We consider a $d$-dimensional representation as in (\ref{algrep}), in $\boldsymbol{M}=\operatorname{Mat}_{d\times d}(\boldsymbol{K}_s)$. 
We consider the following element of $
K_\infty\widehat{\otimes}_{\FF_q}\boldsymbol{M}=\operatorname{Mat}_{n\times n}(K_{s,\infty})\subset \CC_\infty\widehat{\otimes}_{\FF_q}\boldsymbol{M}=\operatorname{Mat}_{n\times n}(\mathbb{K}_s)$,
the topological product $\widehat{\otimes}_{\FF_q}$ being considered with respect to the 
trivial norm over $\boldsymbol{M}$. We denote by $\Pi_\sigma$ the convergent product
$$\Pi_\sigma=\prod_{i\geq 0}(I_d-\sigma(\theta)\theta^{-q^i})^{-1}\in \operatorname{GL}_d(K_{s,\infty})$$
(where $I_d$ denotes the identity matrix). Let $\lambda_\theta$ be a root $(-\theta)^{\frac{1}{q-1}}\in K^{ac}$ of $-\theta$. The {\em $\omega$-value} associated to $\sigma$ is the product
$$\omega_\sigma=\lambda_\theta\Pi_\sigma\in\operatorname{GL}_d(K_{s,\infty}(\lambda_\theta)).$$ We have, on the other hand, 
a continuous $\boldsymbol{K}_s$-algebra automorphism $\tau$ of $\mathbb{K}_s$
uniquely defined by setting $\tau(\theta)=\theta^q$. By using an ultrametric version of 
Mittag-Leffler decomposition, it is easy to show that $\mathbb{K}_s^{\tau=1}$, the subfield of
$\mathbb{K}_s$ of the $\tau$-invariant elements, is equal to $\boldsymbol{K}_s$. We denote by $\tau$
the algebra endomorphism of $\operatorname{Mat}_{d\times d}(\mathbb{K}_s)$
defined by applying $\tau$ entry-wise.

\begin{Lemma}\label{equationtau}
The element $\omega_\sigma$ is a generator of the free $\boldsymbol{M}$-submodule of  rank one of 
$\operatorname{Mat}_{d\times d}(\mathbb{K}_s)$ of the solutions of the $\tau$-difference equation
$$\tau(X)=(\sigma(\theta)-\theta I_d)X.$$\end{Lemma}
\begin{proof} (Sketch.) It is easy to verify that $\omega_\sigma$ is a solution of the above equation. If $\omega'$
is another solution, then $Y=\omega'\omega_\sigma^{-1}$ is solution of $\tau(Y)=Y$ in $\operatorname{Mat}_{d\times d}(\mathbb{K}_s)$, hence, $Y\in\boldsymbol{M}$.
\end{proof}
\begin{Remark}{\em It is also easy to prove that, for $\sigma$ as in (\ref{algrep}),
$\det(\omega_\sigma)=\omega_\alpha$, where $\alpha\in\boldsymbol{K}_s[\theta]$
is $-1$ times the characteristic polynomial of $\sigma(\theta)$ and $\omega_\alpha$ is  the 
function defined in \cite[\S 6]{APTR}.}\end{Remark} Let $T$ be another indeterminate.
The algebra $\operatorname{Mat}_{d\times d}(\mathbb{K}_s)$ is endowed with 
a structure of $A[T]$-module in two ways. The first structure is that in which 
the multiplication by $\theta$ is given by the usual diagonal multiplication, and the multiplication by $T$
is given by the left multiplication by $\sigma(\theta)$; this defines indeed, uniquely, a module structure. The second structure, called {\em Carlitz module structure}, denoted by 
$C(\operatorname{Mat}_{d\times d}(\mathbb{K}_s))$, has the same multiplication by $T$
and has the multiplication $C_{\theta}$ by $\theta$ independent of the choice of $\sigma$, and defined as follows. If $m\in C(\operatorname{Mat}_{d\times d}(\mathbb{K}_s))$, then $C_\theta(m)=\theta m+\tau(m)$.

We have the exponential map $$\exp_C:\operatorname{Mat}_{d\times d}(\mathbb{K}_s)\rightarrow
C(\operatorname{Mat}_{d\times d}(\mathbb{K}_s))$$
defined by $\exp_C(f)=\sum_{i\geq 0}D_i^{-1}\tau^i(f)$, where $D_i$ is the product of the monic 
polynomials of $A$ of degree $i$. It is quite standard  to check that this is a continuous, open, surjective $A[T]$-module homomorphism,
of kernel $\widetilde{\pi}\operatorname{Mat}_{d\times d}(\boldsymbol{K}_s[\theta])$,
where 
$$\widetilde{\pi}:=\theta\lambda_\theta\prod_{i>0}(1-\theta^{1-q^i})^{-1}\in\lambda_\theta K_\infty\subset \CC_\infty$$ is a fundamental period of Carlitz's exponential $\exp_C:\CC_\infty\rightarrow\CC_\infty$.
\begin{Lemma}\label{lemmaexpmatrix}
We have $\omega_\sigma=\exp_C\left(\widetilde{\pi}(\theta I_d-\sigma(\theta))^{-1}\right).$
\end{Lemma}
\begin{proof} We set $f=\exp_C\left(\widetilde{\pi}(\theta I_d-\sigma(\theta))^{-1}\right)$.
Since $(C_{\theta}-\sigma(\theta))(f)=0$ in $C(\operatorname{Mat}_{d\times d}(\mathbb{K}_s))$, Lemma \ref{equationtau} tells us that $f$ belongs to 
the free $\boldsymbol{M}$-submodule of  rank one of 
$\operatorname{Mat}_{d\times d}(\mathbb{K}_s)$ of the solutions of the homogeneous linear difference equation described in that statement. Now, observe that 
$$\omega_\sigma=\theta\lambda_\theta(\theta I_d-\sigma(\theta))^{-1}+M_1,\quad f=\theta\lambda_\theta(\theta I_d-\sigma(\theta))^{-1}+M_2,$$ where $M_1,M_2$ are matrices with coefficients in $K_{s,\infty}$ whose entries have Gauss norms $<|\lambda_\theta|=q^{\frac{1}{q-1}}$. Hence $\omega_\sigma=f$.
\end{proof}
\subsubsection{$L$-values associated to a semi-character}\label{Lvalues}

we again suppose that $\boldsymbol{M}=\operatorname{Mat}_{d\times d}(\boldsymbol{K}_s)$, with $\boldsymbol{K}_s$ as in the previous sections.
Let $\sigma$ be a semi-character $A^+\rightarrow\boldsymbol{M}$
Let $n$ be a positive integer. The {\em $n$-th $L$-value} associated to $\sigma$ is the 
following element of $\operatorname{GL}_d(K_{s,\infty})$:
$$L_\sigma(n)=\prod_{P}\left(I_d-\sigma(P)P^{-n}\right)^{-1}=\sum_{a\in A^+}\sigma(a)a^{-n}=I_d+\cdots,$$
the product running over the irreducible elements of $A^+$.
\subsubsection{Determinant} We write $\sigma=\sigma_1\cdots\sigma_s$ for injective $\FF_q$-algebra homomorphisms
$\sigma_i:A\rightarrow\boldsymbol{M}$ with $\sigma_i(\theta),\sigma_j(\theta)$ commuting each other. 
The elements $L_\sigma(n)$ and $\omega_{\sigma_1}\cdots\omega_{\sigma_s}$ commute each other.
Further we denote by $\lambda_{i,1},\ldots,\lambda_{i,d}\in\boldsymbol{K}_s^{ac}$ the eigenvalues of
$\sigma_i(\theta)$ for $i=1,\ldots,s$ (considered with multiplicities). For simplicity, we suppose that none of these eigenvalues belong to $\FF_q^{ac}$. On the other hand, we consider variables
$x_1,\ldots,x_s$ and the $L$-value:
$$\mathcal{L}_s(n):=\prod_{P}(1-\psi_s(P)P^{-n})^{-1},$$
where $\psi:A^+\rightarrow\FF_q[x_1,\ldots,x_s]^*$ is the semi-character defined by $a\mapsto a(x_1)\cdots a(x_s)$, and the series converges in the completion of $K[x_1,\ldots,x_s]$ for the Gauss norm extending 
$|\cdot|$.
\begin{Lemma}
We have the formula
$$\det(L_\sigma(n))=\mathcal{L}_s(n)_{x_j=\lambda_{i,j}\atop j=1,\ldots,s}\in K_{s,\infty}.$$
\end{Lemma}
\begin{proof}
We note that, for every polynomial $P\in A^+$, 
$$\det((I_d-\sigma(P)P^{-n})^{-1})=P^{dn}\det(I_dP^n-\sigma(P))^{-1}.$$ By the well known properties of the characteristic polynomial of an endomorphism, we have that 
$$\det(P^n-\sigma(P))=\prod_{i=1}^d(X-\mu_{i,P})_{X=P^n},$$
where $\mu_{i,P}\in\boldsymbol{K}_s^{ac}$ are the eigenvalues of the left multiplication by $\sigma(P)$. Now, observe that
$$\sigma(P)=\prod_{j=1}^s\sigma_j(P)=\prod_{j=1}^sP(\sigma_j(\theta))$$
(the elements $\sigma_j(\theta)$ commute each other). Hence, 
$\mu_{i,P}=\prod_{j=1}^sP(\lambda_{i,j})$ for all $i=1,\ldots,d$. Thus,
$$\det(I_d-\sigma(P)P^{-n})^{-1}=P^{dn}\prod_{j=1}^s\left(I_dP^n-\prod_{j=1}^sP(\lambda_{i,j})\right)^{-1},$$
and the lemma follows.
\end{proof}
\subsubsection{The case $n=1$} We write $\sigma=\sigma_1\cdots\sigma_s$ as above. Since for all $a\in A^+$,
$C_a(\omega_{\sigma_i})=\sigma_i(a)\omega_{\sigma_i}\in\boldsymbol{M}$, we have the convergent series identity
$$\sum_{a\in A^+}a^{-1}C_a(\omega_{\sigma_1})\cdots C_a(\omega_{\sigma_s})=L_\sigma(1)\omega_{\sigma_1}\cdots\omega_{\sigma_s}\in \operatorname{Mat}_{d\times d}(K_{s,\infty}(\lambda_\theta))$$ (in fact, the series converges to an invertible matrix).
\subsubsection{A Simple application of Anderson's log-algebraic Theorem}
We now invoke the result of B. Angl\`es, F. Tavares Ribeiro and the author \cite[Theorem 8.2]{APTR} (note that in the statement, we can set $Z=1$). For completeness, we mention the following result, which 
is a very easy consequence of ibid.:
\begin{Proposition}\label{Taelmanunits}
For every semi-character $\sigma:A^+\rightarrow\boldsymbol{M}$ with $\sigma=\sigma_1\cdots\sigma_s$ as above,
$$(\omega_{\sigma_1}\cdots\omega_{\sigma_s})^{-1}\exp_C(\omega_{\sigma_1}\cdots\omega_{\sigma_s}L_\sigma(1))=:S_\sigma\in \operatorname{Mat}_{d\times d}(\boldsymbol{K}_s[\theta]).$$
Further, if $s\equiv1\pmod{q-1}$ and if $s>1$, the matrix with polynomial entries $S_\sigma$ is zero.
In particular, in this case,
$$L_\sigma(1)=\widetilde{\pi}(\omega_{\sigma_1}\cdots\omega_{\sigma_s})^{-1}\mathbb{B}_\sigma,$$
where $\mathbb{B}_\sigma$ is a matrix with polynomial entries in $\boldsymbol{K}_s[\theta]$.
\end{Proposition}
Hence, $L_\sigma(1)$ is a "Taelman unit", in the sense of \cite{TAE2}.
If $s=1$, we have a more explicit property. In this case, $\sigma$ extends to an
algebra homomorphism $\sigma:A\rightarrow\boldsymbol{M}$, and we have the simple explicit formula
$$L_\sigma(1)=\omega_\sigma^{-1}(I_d\theta-\sigma(\theta))^{-1}\widetilde{\pi}$$ which 
can be proved in a way very similar to that of \cite[\S 4]{Pe}. We are going to see, 
in \S \ref{eisensteinseries} that $L_\sigma(n)$ is related to certain vectorial Eisenstein series, when $n\equiv s\pmod{q-1}$.

\subsection{Representations of $\Gamma$ associated to an algebra representation}
Let $\boldsymbol{K}$ be any commutative field extension of $\FF_q$.
Let $\sigma$ be a $d$-dimensional representation as in (\ref{algrep}). 
We associate to it, canonically,
a representation of $\Gamma=\operatorname{GL}_2(A)$ in $\operatorname{GL}_{2d}(\boldsymbol{K})$.

We consider the map
$\Gamma\xrightarrow{\rho_\sigma}\operatorname{Mat}_{2d\times 2d}(\boldsymbol{K}),$
defined by $$\rho_\sigma\left(\bigsqm{a}{b}{c}{d}\right)=\bigsqm{\sigma(a)}{\sigma(b)}{\sigma(c)}{\sigma(d)}.$$
Then, $\rho_\sigma$ determines a representation 
$\Gamma\rightarrow\GL_{2d}(\boldsymbol{K}).$ Indeed, $\sigma(a),\sigma(b)$ commute each other, for 
$a,b\in A$.
Furthermore, we have:
\begin{Lemma}\label{lemma2}
$\rho_\sigma$ is irreducible if and only if $\sigma$ is irreducible.
\end{Lemma}
\begin{proof}
Let $V$ be a non-trivial sub-vector space of $\operatorname{Mat}_{d\times 1}(\boldsymbol{K})$
such that $\sigma(a)(V)\subset V$. Then, if $\gamma=(\begin{smallmatrix} a & b \\ c & d \end{smallmatrix})\in\Gamma$, we have $$\rho_\sigma(\gamma)=\begin{pmatrix} \sigma(a) & \sigma(b) \\ \sigma(c) & \sigma(d) \end{pmatrix}(V\oplus V)\subset V\oplus V$$
and $\sigma$ not irreducible implies $\rho_\sigma$ not irreducible.

Now, let us assume that $\sigma$ is irreducible and let us consider $V$ a non-zero sub-vector space
of $\operatorname{Mat}_{2d\times 1}(\boldsymbol{K})$ which is $\rho_\sigma$-invariant. We observe that
$V\cap\Delta\neq\{0\}$, with $\Delta=\{\binom{v}{v}:v\in\operatorname{Mat}_{d\times 1}(\boldsymbol{K})\}$.
Indeed, $\rho_\sigma((\begin{smallmatrix} 0 & 1 \\ 1 & 0 \end{smallmatrix}))=(\begin{smallmatrix} 0 & I_d \\ I_d & 0 \end{smallmatrix})$ and $\rho_\sigma((\begin{smallmatrix} 1 & 0 \\ 0 & -1 \end{smallmatrix}))=(\begin{smallmatrix} I_d & 0 \\ 0 & -I_d \end{smallmatrix})$. 
If $\binom{x}{y}$ is a non-zero vector of $V$ with $x\neq-y$, we have $\binom{x}{y}+\rho_\sigma((\begin{smallmatrix} 0 & 1 \\ 1 & 0 \end{smallmatrix}))\binom{x}{y}=\binom{x+y}{x+y}\in\Delta\setminus\{0\}$. If $x=-y$, then $\rho_\sigma((\begin{smallmatrix} 1 & 0 \\ 0 & -1 \end{smallmatrix}))\binom{x}{y}\in\Delta\setminus\{0\}$. Let $\binom{v}{v}$ be non-zero in $V\cap\Delta$. Since for all $a\in A$,
$\rho_\sigma((\begin{smallmatrix} 1 & a \\ 0 & 1 \end{smallmatrix}))\binom{v}{v}=(\begin{smallmatrix} I_d & \sigma(a) \\ 0 & I_d \end{smallmatrix})\binom{v}{v}=\binom{\sigma(a')v}{v}$ with $a'=a+1$ we have $\{\binom{\sigma(a)(v)}{v}:a\in A\}\subset V$.
Let $W$ be the 
$\boldsymbol{K}$-sub-vector space of $\operatorname{Mat}_{d\times 1}(\boldsymbol{K})$
generated by the set $\{\sigma(a)(v):a\in A\}$. Then, $W$ is $\sigma$-invariant: if 
$w=\sum_ic_i\sigma(a_i)(v)\in W$ ($c_i\in \boldsymbol{K}$, $a_i\in A$), we have that
$$\sigma(a)(w)=\sum_ic_i\sigma(aa_i)(v)\in W.$$ By hypothesis, $W$ is non-zero, so that
$W=\operatorname{Mat}_{d\times 1}(\boldsymbol{K})$. We have proved that 
$V\supset \operatorname{Mat}_{d\times 1}(\boldsymbol{K})\oplus\{v\}$. Translating, this means that 
$V\supset\operatorname{Mat}_{d\times 1}(\boldsymbol{K})\oplus\{0\}$. Applying $\rho_\sigma((\begin{smallmatrix} 0 & 1 \\ 1 & 0 \end{smallmatrix}))$ we see that
$V\supset\{0\}\oplus\operatorname{Mat}_{d\times 1}(\boldsymbol{K})$ and $V=\operatorname{Mat}_{2d\times 1}(\boldsymbol{K})$.\end{proof}

\subsubsection*{Example}
We can construct, in particular, the representation $\rho_P=\rho_{\sigma_P}:\Gamma\rightarrow\GL_{2d}(\boldsymbol{A})$ which is 
irreducible if and only if $P$ is irreducible. 

\section{Symmetric powers}\label{symmetricpowers}

In the first part of this section, we suppose that $q$ is a prime number; $p=q$. Let $B$ be an $\FF_p$-algebra.
We denote by $\rho$ the tautological representation $\GL_2(B)\rightarrow\operatorname{GL}_2(B)$.
We consider the representation $$\rho_{r}=\operatorname{Sym}^r(\rho):\GL_2(B)\rightarrow\operatorname{GL}_{r+1}(B),$$ where $\operatorname{Sym}^r$ denotes the $r$-th symmetric power realized in the space of polynomials homogeneous of degree $r+1$ with
coefficients in $B$. If 
$\gamma=\sqm{a}{b}{c}{d}\in\GL_2(B)$, then
$$\rho_{r}(\gamma)(X^{r-i}Y^i)=(aX+cY)^{r-i}(bX+dY)^i,\quad i=0,\ldots, r.$$

Associated to
an integer $l\geq 0$ with $p$-expansion $l=l_0+l_1p+\cdots+l_sp^s$ ($0\leq l_i\leq p-1$), we also consider the representations
$$\rho^\irr_{l}=\rho_{l_0}\otimes\rho_{l_1}^{(1)}\otimes\cdots\otimes\rho_{l_s}^{(s)},$$
where, for a matrix $M$ with entries in $B$, $M^{(i)}$ denotes the matrix obtained from $M$ raising all its entries to the power $p^i$.
The dimension of $\rho_l^\irr$ is equal to 
$$\phi_p(l)=\prod_i(l_i+1).$$

\begin{Lemma}\label{isomorphictoasub}
The representation $\rho^\irr_l$ is isomorphic to a sub-representation of $\rho_l$.
\end{Lemma}
\begin{proof}
We actually construct the sub-representation explicitly; the Lemma will follow easily. 
We consider, for $\gamma\in\GL_2(B)$, the matrix $\rho_{l}^{\irrr}(\gamma)$ which 
is the square matrix extracted from $$\rho_{l}(\gamma)=(\rho_{i,j})_{1\leq i,j\leq l+1}$$ in the following way. 
If $0\leq r\leq l$ is such that $\binom{l}{r}\equiv0\pmod{p}$, we drop the $(r+1)$-th row and 
the $(r+1)$-th column. In other words, one uses the row matrix
$$\mathcal{D}_l=\left(\binom{l}{l},\ldots,\binom{l}{r},\ldots,\binom{l}{0}\right)$$
and discards rows and columns of $\rho_{l}$ according with the vanishing of the corresponding 
entry of $\mathcal{D}_l$ and what is left precisely defines the matrix $\rho_l^\irrr$. By Lucas
formula, $\rho_{l}^{\irrr}$ has dimension $\phi_p(l)$ and it is easy to see, by induction on the number of digits of 
the $p$-expansion of $l$, that
$\rho_l^\irrr\cong\rho_l^\irr$. 
\end{proof}
\subsubsection*{Example} If $l=1+p$, we have
$$\mathcal{D}_l^\irr=(1,1,1,1)=\left(\binom{l}{0},\binom{l}{1},\binom{l}{p},\binom{l}{l}\right).$$
In this case, we find, for $\gamma=\sqm{a}{b}{c}{d}$,
$$\rho_{l}^{\irrr}(\gamma)=\left(\begin{array}{llll} a^{p+1} & a^pb & ab^p & b^{p+1}\\
                                        a^pc & a^pd & b^p c & b^p d\\ ac^
                                p & bc^p & a d^p & bd^p \\
                  c^{p+1} & c^pd & c d^p & d^{p+1}
        \end{array}\right).$$

\begin{Remark}{\em 
We notice the following algorithm to construct the sequence of dimensions $(\phi_p(l))_{l\geq 1}$.
Define $a_0=(1)$, $a_1=(1,2,\ldots,p)$ (equal to the concatenation $[a_0,2a_0,\ldots,pa_0]$) and then,
inductively, $$a_n=[a_{n-1},2a_{n-1},\ldots,pa_{n-1}].$$ Since it is clear that for all $n$, $a_{n-1}$ is a prefix
of $a_n$, there is a well defined inductive limit $a_\infty$ of the sequence $(a_n)_{n\geq0}$ which is easily 
seen to be
equal to the sequence $(\phi_p(l))_{l\geq 1}$.}\end{Remark}

\subsection{Representations of $\operatorname{SL}_2(\mathbb{F}_{q'})$.} \label{section1}

Let us set $q'=p^f$ with $f> 0$. Then, $B=\FF_{q'}$ is an $\FF_p$-algebra
and we can construct the representations $\rho_l$ and $\rho_l^\star\cong\rho_L^I$ of the beginning of this section.
We denote by $\overline{\rho}_l$ the representation $\rho_l^\irr$ with $B=\mathbb{F}_{q'}$ restricted to $\SL_2(B)$.
By using the fact that any non-zero stable subspace in a representation of a $p$-group over a vector space has a non-zero fixed vector, it is easy to show and in fact well known 
that, for all $l\geq0$, $\overline{\rho}_{l}$ is an irreducible representation if and only if $l<q'$. By Schur's theory, one shows that
the representations $\overline{\rho}_l$ with $l<q'$ exhaust all the isomorphism classes of irreducible representations of $\operatorname{SL}_2(\mathbb{F}_{q'})$ over $\FF_p^{ac}$.
Indeed, counting isomorphism classes of $\operatorname{SL}_2(\mathbb{F}_{q'})$ is an easy task and we know that their number coincides
with the number of isomorphism classes of irreducible representations so it suffices to check that the representations above are
mutually inequivalent which is an elementary task. This explicit description first appears in the 
paper \cite{BN} of Brauer and Nesbitt. Steinberg tensor product theorem 
\cite[Theorem 16.12]{MAL&TES} provides such a description when, at the place of $G=\SL_2$,
we have, much more generally, a semisimple algebraic group of simply connected type, defined over an algebraically closed field $B$ of positive characteristic. This also implies Lemma \ref{isomorphictoasub}. The author is thankful to Gebhard B\"ockle for having drawn his attention to this result and reference.

\subsection{Some representations of $\operatorname{GL}_2(A)$.} \label{section2}
We now set $q=p^e$ with $e>0$. We also set $\boldsymbol{K}:=\FF_q(t)$ for a variable $t$
all along this subsection. We consider the algebra homomorphism $\chi_t:A\rightarrow\FF_q[t]\subset\boldsymbol{K}$
defined by 
$\chi_t(a)=a(t)=a_0+a_1t+\cdots+a_dt^d$ for $a=a_0+a_1\theta+\cdots+a_d\theta^d\in A$, with coefficients 
$a_i\in\FF_q$. We extend our notations by setting, for a matrix $M$ with entries in $A$, $\chi_t(M)$ the 
matrix obtained applying $\chi_t$ entry-wise.
We denote by $\rho_{t},\rho_{t,l},\rho_{t,l}^\irr$ the representations
$\chi_t\circ\rho_1,\chi_t\circ\rho_l,\chi_t\circ\rho_l^\irr$ over $\boldsymbol{K}$-vector spaces with the appropriate dimensions, of the group $\Gamma=\GL_2(A)$.

\begin{Lemma}\label{theoprinc} For all $l$ as above, 
the representation $\rho_{t,l}^\irr$ is irreducible.
\end{Lemma}
\begin{proof} It suffices to show that the restriction to $\SL_2(\FF_p[\theta])\subset\Gamma$ is irreducible. Let us consider an element $\zeta\in\mathbb{F}_q^{ac}$ of degree $f$ and let us denote by 
$\mathbb{F}_{q'}$ with $q'=p^f$ the subfield $\mathbb{F}_p(\zeta)$ of $\FF_p^{ac}$. 
The group homomorphism $$\operatorname{ev}_\zeta:\SL_2(\FF_p[\theta])\rightarrow\operatorname{GL}_2(\mathbb{F}_{q'})$$
defined by the entry-wise evaluation $\operatorname{ev}_\zeta$ of $\theta$ by $\zeta$ has image $\operatorname{SL}_2(\mathbb{F}_{q'})$. Indeed, the evaluation map $\operatorname{ev}_\zeta:\FF_p[\theta]\rightarrow\FF_{q'}$ is surjective, the image of 
$\operatorname{SL}_2(\FF_p[\theta])$ by $\operatorname{ev}_\zeta$ clearly 
contains the subgroup of triangular upper and lower matrices with coefficients in $\mathbb{F}_{q'}$,
which are known to generate $\operatorname{SL}_2(\mathbb{F}_{q'})$ 

We set $N=\phi_p(l)$.
Let $V$ be a non-zero $\boldsymbol{K}$-subvector space of $\boldsymbol{K}^{N}$ which is stable under the 
action of the representation of  $\SL_2(\FF_p[\theta])$ induced by $\rho^\irr_{t,l}$. Let us fix a basis $b$ of $V$. We choose $f$ big enough so that $q'>l,q$ and the image $b'$ of $b$ in $\FF_{q'}^N$ by the evaluation at $t=\zeta$ is well defined 
and non-zero. Then, the $\FF_{q'}$-span of $b'$ is a non-trivial sub-vector space of $\FF_{q'}^N$
which is left invariant under the action of $\overline{\rho}_l$, which is impossible.  
\end{proof}
\begin{Remark}{\em 
Let $m$ be a class of $\ZZ/(q-1)\ZZ$ and let us consider the representation 
$$\rho^\irr_{t,l,m}:\Gamma\rightarrow\operatorname{GL}_{\phi_p(l)}(\boldsymbol{K})$$
defined by 
$$\rho^\irr_{t,l,m}:=\rho_{t,l}^{\irr}\otimes\sideset{_{}^{}}{_{}^{-m}}\det.$$
By Lemma \ref{theoprinc}, it is irreducible.
However, the representations $\rho^\irr_{t,l,m}$ do not cover all the irreducible representations
of $\Gamma$ in $\GL_N(\boldsymbol{K})$ for some $N$. Due to the fact that 
we evaluate the functor $\GL_N$ on a ring which is not a field (here, the ring $A$), there are irreducible representations
which, after specialization at roots of unity, do not give irreducible representations of $\SL_2(\FF_{q'})$.}\end{Remark}

\begin{Remark}{\em The group $S_{(p)}$ of $p$-adic digit permutations of $\ZZ_p$ discussed by Goss in \cite{Go} acts on the positive integers $l$
by means of their expansions in base $p$. This defines an action of the group $S_{(p)}$
on the set of representations $\rho_l^\irr$ by $\nu(\rho_l^\irr)=\rho^\irr_{\nu(l)}$, for $\nu\in S_{(p)}$.
Note that the dimensions of these representations are $S_{(p)}$-invariants.
It is easy to show that $\nu(\rho^\irr_l)\cong\rho^\irr_{l'}$ if and only if $\nu(l)=l'$.}\end{Remark}

\begin{Remark}{\em 
We are thankful to Pietro Corvaja for having pointed out the following property.
{\em Let $k$ be a perfect field. Then, for all $\gamma\in\operatorname{SL}_2(k^{ac})$ there exists a morphism
$\phi:\mathbb{A}^1\rightarrow\operatorname{SL}_2$ defined over $k$ and $\alpha\in k^{ac}$, such that 
$\phi(\alpha)=\gamma$.}
}\end{Remark}

\section{Products of representations}\label{tensorprod}

Let $t_1,\ldots,t_s$ be independent variables. We denote by $\underline{t}_s$ the set of
variables $(t_1,\ldots,t_s)$ and we set $\boldsymbol{K}_s=\FF_q(\underline{t}_s)$. 
If $s=1$, we write $t=t_1$ and we have $\boldsymbol{K}_s=\boldsymbol{K}$, the field of \S \ref{section2}.
We also consider $\underline{l}=(l_1,\ldots,l_s)$ an $s$-tuple 
with entries in $\ZZ$ which are $\geq 1$.
\begin{Theorem}\label{corollaryt1ts}
the representation
$$\rho^{\irrbis}_{\underline{t},\underline{l}}:=\rho_{t_1,l_1}^\irr\otimes\cdots\otimes\rho_{t_s,l_s}^\irr:\Gamma\rightarrow\operatorname{GL}_{\phi_p(l_1)\cdots\phi_p(l_s)}(\boldsymbol{K}_s)$$ is irreducible.
\end{Theorem}
\begin{proof}
We set $N=\phi_p(l_1)\cdots\phi_p(l_s)$.
Let us suppose by contradiction that the statement is false. Then, there exists a $\boldsymbol{K}_s$-sub-vector space
$V\neq\{0\}\subset\boldsymbol{K}_s^N$ such that for all $\gamma\in\Gamma$, $\rho^{\irrbis}_{\underline{t},\underline{l}}(\gamma)(V)\subset V$. Let us fix a basis $v=(v_1,\ldots,v_r)$ of $V$. 
For integers $0\leq k_1\leq\cdots\leq k_s$, we denote by $\operatorname{ev}$ the map $\FF_q[\underline{t}_s]\rightarrow \FF_q[t]$
which sends $a(t_1,\ldots,t_s)\in \FF_q[\underline{t}_s]$ to $a(t^{k_1},\ldots,t^{k_s})\in \FF_q[t]$.
This map is a ring homomorphism whose kernel is the prime ideal $\mathcal{P}$ generated by the polynomials 
$t_j-t_{j-1}^{q^{k_j-k_{j-1}}}$, $j=2,\ldots,s$. We consider the associated multiplicative set $S=\FF_q[\underline{t}_s]\setminus\mathcal{P}$. Then, the evaluation map $\operatorname{ev}$ extends to $S^{-1}\FF_q[\underline{t}_s]$
which is Zariski dense in $\boldsymbol{K}_s=\FF_q(\underline{t}_s)$. We now extend $\operatorname{ev}$ coefficient-wise
on every matrix, vector, etc. with entries in $S^{-1}\FF_q[\underline{t}_s]$. If $k_1$ is big enough,
$\operatorname{ev}(v)$ is well defined and non-zero. 

We can in fact choose $k_1,\ldots,k_s$ so that we also have at once, $\operatorname{ev}(\rho^{\irrbis}_{\underline{t},\underline{l}})=
\rho_{t,l}^\irr$ for some $l\geq 0$. Indeed, if we write the $p$-expansions $l_i=l_{i,0}+l_{i,1}p+\cdots+l_{i,r}p^r$
($i=1,\ldots,s$) for some $r\geq 0$, then we can choose $k_1,\ldots,k_s$ so that 
there is no carry over in the $p$-expansion of the sum $l=l_1q^{k_1}+l_2q^{k_2}+\cdots+l_sq^{k_s}$; for such a
choice of $k_1,\ldots,k_s$, $\operatorname{ev}(\rho^{\irrbis}_{\underline{t},\underline{l}})$ is thus irreducible.

We now set $W$ to be the $\boldsymbol{K}$-span of $\operatorname{ev}(v)$, well defined and non-trivial
in $\boldsymbol{K}^N$ (we recall that $\boldsymbol{K}=\FF_q(t)$).
Let $w$ be in $W$. We can write $w=a_1\operatorname{ev}(v_1)+\cdots+a_r\operatorname{ev}(v_r)$
for elements $a_i\in\boldsymbol{K}$. Then, $$\rho_{t,l}^\irr(\gamma)(w)=a_1\rho_{t,l}^\irr(\gamma)(\operatorname{ev}(v_1))+\cdots+a_r\rho_{t,l}^\irr(\gamma)(\operatorname{ev}(v_r))=
a_1\operatorname{ev}(\rho_{\underline{t},\underline{l}}^{\irrbis}(\gamma)(v_1))+\cdots+a_r\operatorname{ev}(\rho_{\underline{t},\underline{l}}^{\irrbis}(\gamma)(v_r))$$ is a vector of $W$, hence contradicting 
the irreducibility of $\rho_{t,l}^\irr$.
\end{proof}

\section{Applications to Poincar\'e series}\label{poincare}

\begin{Definition}
{\em We say that a representation $\Gamma\xrightarrow{\rho}\GL_N(\boldsymbol{K})$
is {\em normal to the depth} $L\in\{1,\ldots,N\}$ if for all 
$\gamma\in H=\{\sqm{*}{*}{0}{1}\}\subset\Gamma$, we have that $\rho(\gamma)=\sqm{*}{*}{0}{I_L}$,
where $I_L$ denotes the identity matrix of size $L$.}
\end{Definition}
A representation as above which is normal to the depth $N$ has finite image. To see this,
note that $\rho((\begin{smallmatrix} * & * \\ 0 & * \end{smallmatrix}))=((\begin{smallmatrix} * & 0 \\ 0 & * \end{smallmatrix}))$ is finite. Hence, 
\begin{eqnarray*}
\lefteqn{\rho\begin{pmatrix} * & 0 \\ * & * \end{pmatrix}=
\rho\left(\begin{pmatrix} 0 & 1 \\ 1 & 0 \end{pmatrix}\begin{pmatrix} * & * \\ 0 & * \end{pmatrix}\begin{pmatrix} 0 & 1 \\ 1 & 0 \end{pmatrix}\right)=}\\
&=&\rho\begin{pmatrix} 0 & 1 \\ 1 & 0 \end{pmatrix}\rho\begin{pmatrix} * & * \\ 0 & * \end{pmatrix}\rho\begin{pmatrix} 0 & 1 \\ 1 & 0 \end{pmatrix}\\
&=&\rho\begin{pmatrix} 0 & 1 \\ 1 & 0 \end{pmatrix}\rho\begin{pmatrix} * & 0 \\ 0 & * \end{pmatrix}\rho\begin{pmatrix} 0 & 1 \\ 1 & 0 \end{pmatrix}\\
&=&\rho\left(\begin{pmatrix} 0 & 1 \\ 1 & 0 \end{pmatrix}\begin{pmatrix} * & 0 \\ 0 & * \end{pmatrix}\begin{pmatrix} 0 & 1 \\ 1 & 0 \end{pmatrix}\right)=\rho\begin{pmatrix} * & 0 \\ 0 & * \end{pmatrix}.
\end{eqnarray*}
We thus have $\rho(\begin{smallmatrix} * & 0 \\ * & * \end{smallmatrix})=\rho(\begin{smallmatrix} * & * \\ 0 & * \end{smallmatrix})$ finite, and $\rho(\Gamma)=\rho((\begin{smallmatrix} * & * \\ 0 & * \end{smallmatrix})(\begin{smallmatrix} * & 0 \\ * & * \end{smallmatrix}))=\rho(\begin{smallmatrix} * & 0 \\ 0 & * \end{smallmatrix})$ is finite. 

The representation $\Gamma\xrightarrow{\rho_\sigma}\GL_N(\boldsymbol{K})$ with $N=2d$
associated to an algebra representation $A\xrightarrow{\sigma}\operatorname{Mat}_{d\times d}(\boldsymbol{K})$ is normal to the depth $L=d$.

If, for some ring $R$, we have that $M\in\operatorname{Mat}_{N\times N}(R)=\sqm{*}{*}{X}{Y}$
with $Y\in\operatorname{Mat}_{L\times L}$, we set $M_L=(X,Y)\in\operatorname{Mat}_{L\times N}(R)$.
In other words, $M_L$ is the matrix constituted by the last $L$ lines of $M$.

We denote by $\Omega=\CC_\infty\setminus K_\infty$ the Drinfeld "upper-half plane" of $\CC_\infty$. 
We choose $m$ a non-negative integer, an integer $w\in\ZZ_{>0}$, and, for $\delta\in\Gamma$ and 
$z\in\Omega$, we set $\mu_{w,m}(\delta,z)=\det(\delta)^{-m}J_\delta(z)^w$, where 
$J_\gamma(z)$ is the usual "Drinfeldian" factor of automorphy defined, for $\gamma=\sqm{a}{b}{c}{d}\in\Gamma$
by $J_\gamma(z)=cz+d$. We also denote by $u(z)$ the uniformizer at infinity of $\Omega$, that is,
the function $u(z)=\frac{1}{e_C(z)}$ with $e_C$ the exponential function $\CC_\infty\rightarrow\CC_\infty$
with lattice period $A\subset\CC_\infty$. 

We consider a representation $\Gamma\xrightarrow{\rho}\GL_N(\boldsymbol{K})$, normal to the depth $L$.
Following \cite[\S 2.4]{Pe}, we set, for $\delta\in\Gamma$,
$$f_\delta=\mu_{w,m}(\delta,z)^{-1}u^m(\delta(z))\rho(\delta)_L:\Omega\rightarrow\operatorname{Mat}_{L\times N}(\mathbb{K}),$$
where we recall that $\mathbb{K}=\mathbb{K}_1$ is the completion of $\CC_\infty(t)$
for the Gauss norm. It is easy to show that the series
$$\mathcal{E}_{w,m,\rho}(z)=\sum_{\delta\in H\backslash \Gamma}f_\delta,$$ the sum being over the 
representatives of the cosets of $\Gamma$ modulo the left action of $H=\{\sqm{*}{*}{0}{1}\}$,
converges to a holomorphic function
$$\mathcal{E}_{w,m,\rho}:\Omega\rightarrow\operatorname{Mat}_{L\times N}(\mathbb{K}),$$
 in the sense of \cite{PEL&PER}. 

\begin{Remark}{\em 
For convenience of the reader, we recall here the definition of an holomorphic function $\Omega\rightarrow\mathbb{K}$. For $z\in\Omega$, we set $|z|_\Im:=\inf_{\lambda\in K_\infty}|z-\lambda|$,
which is non-zero. We also define, on $\Omega$, a Stein-like structure by considering 
the affinoids $U_n=\{z\in\Omega;|z|\leq q^n\text{ and }|z|_\Im\geq q^{-n}\}$, so that $\Omega=\cup_{n\in\NN}U_n$. For $n$ fixed, a function $f:U_n\rightarrow\mathbb{K}$ is {\em holomorphic}
if it is a uniform limit of a converging sequence of rational functions $U_n\rightarrow\mathbb{K}$, without poles in $U_n$.
A function $f:\Omega\rightarrow\mathbb{K}$ is {\em holomorphic} if, for all $n\geq0$,
the restriction of $f$ to $U_n$ is holomorphic.}\end{Remark}

Following and readapting the proof of \cite[Proposition 22]{Pe}, we obtain:
\begin{Proposition} 
The following properties hold, for $w\in\ZZ_{>0}$, $m\in\ZZ_{\geq0}$ and $\rho$ a representation
$\Gamma\rightarrow\GL_N(\boldsymbol{K})$:
\begin{enumerate}
\item For all $\gamma\in \Gamma$, we have
$$\mathcal{E}_{w,m,\rho}(\gamma(z))=\det(\gamma)^{-m}J_\gamma(z)^w\mathcal{E}_{w,m,\rho}(z)\cdot\rho(\gamma)^{-1},$$
\item There exists 
$h\in\ZZ$ such that $$u(z)^h\mathcal{E}_{w,m,\rho}(z)\rightarrow 0\in\operatorname{Mat}_{L\times N}(\mathbb{K})$$ as $u(z)\rightarrow0$.
\end{enumerate}
\end{Proposition}
The last condition means that $\mathcal{E}_{w,m,\rho}(z)$ is tempered in the sense of \cite{Pe}.
The proposition means that the $L$ columns of the transposed of $\mathcal{E}_{w,m,\rho}$ are vectorial modular forms of weight $w$ and type $m$ with respect to the contragredient representation of $\rho$.
In the next two subsections, we analyze Poincar\'e series associated with two particular classes of representations.

\subsection{Vectorial Poincar\'e series associated to representations $\rho_\sigma$}

Let us consider an irreducible, faithful algebra representation $A\xrightarrow{\sigma}\operatorname{Mat}_{d\times d}(\boldsymbol{K})$, and the associated representation $\Gamma\xrightarrow{\rho_\sigma}\GL_N(\boldsymbol{K})$
with $N=2d$, which is normal to the depth $L=d$. We additionally suppose that the 
characteristic polynomial of $\vartheta=\sigma(\theta)$,
irreducible, is also separable.
\begin{Proposition}\label{proposition1}
The following properties hold, with $\rho=\rho_\sigma$.
\begin{enumerate}
\item If $w-1\not\equiv 2m\pmod{q-1}$, then $\mathcal{E}_{w,m,\rho}=0$, identically.
\item If $w-1\equiv 2m\pmod{q-1}$ and $w\geq (q+1)m+1$, then, the rank of the matrix function $\mathcal{E}_{w,m,\rho}\not=0$ is $d$.
\item In the second case, each row of the matrix function $\mathcal{E}_{w,m,\rho}$
has the entries which are $\mathbb{K}$-linearly independent.
\end{enumerate}
\end{Proposition}

\begin{proof} The hypotheses on $\sigma$ imply that the matrix $\vartheta=\sigma(\theta)$ has distinct
conjugate eigenvalues $\lambda_1,\ldots,\lambda_d\in\boldsymbol{K}^{ac}$ none of which lies in $\FF_q^{ac}$.
We consider a corresponding basis $v_1,\ldots,v_d\in(\boldsymbol{K}^{ac})^d$ of eigenvectors
(considered as column matrices)
which are then common eigenvectors for all the elements of the image of $\sigma$.
In $(\boldsymbol{K}^{ac})^{2d}=(\boldsymbol{K}^{ac})^N$ we consider the basis 
$w_1,\ldots,w_{2d}$ defined by $w_i=v_i\oplus0$  and $w_{d+i}=0\oplus v_i$ for $i=1,\ldots,d$.
We also denote by $M\in\GL_N(\boldsymbol{K}^{ac})$ the matrix whose columns are the $w_i$'s
for $i=1,\ldots,2d$. then, for $\delta=\sqm{a}{b}{c}{d}\in\Gamma$, we have 
$$\rho(\delta)_LM=\delta(t)*(v_1,\ldots,v_d):=(c(\lambda_1)v_1,\ldots,c(\lambda_d)v_d,d(\lambda_1)v_1,\ldots,d(\lambda_d)v_d)\in\operatorname{Mat}_{d\times 2d}(\boldsymbol{K}^{2d}).$$ 
Hence we have, with the same significance of the product $*$ extended linearly, that 
$$M\cdot \mathcal{E}_{w,m,\rho}=\mathcal{E}_{w,m,\chi_t}*(v_1,\ldots,v_d),$$
where $\mathcal{E}_{w,m,\chi_t}:\Omega\rightarrow\operatorname{Mat}_{1\times 2}(\mathbb{K})$
is the function defined by
$$\mathcal{E}_{w,m,\chi_t}(z)=\sum_{\delta=\sqm{a}{b}{c}{d}\in H\backslash \Gamma}
\mu_{w,m}(\delta,z)^{-1}u^m(\delta(z))(c(t),d(t)).$$
This matrix function is the deformation of vectorial Poincar\'e series $\mathcal{E}_{w,m}(z,t)$
considered in \cite[Proposition 22]{Pe} and we know the following:
\begin{itemize}
\item If $w-1\not\equiv 2m\pmod{q-1}$, then $\mathcal{E}_{w,m,\chi_t}$ is identically zero. 
\item If  $w-1\equiv 2m\pmod{q-1}$ and $w\geq (q+1)m+1$, then all the entries of $\mathcal{E}_{w,m,\chi_t}$ are non-zero.
\end{itemize}
We now observe that 
if $C$ is a complete field containing $A$ and if $f(t)=\sum_{i\geq 0}f_it^i\in C[[t]]$ is a non-zero formal power series,
with $f_i\rightarrow0$ for $i\rightarrow\infty$  (an element of the Tate algebra of formal series with coefficients in $C$ in the variable $t$)
then, for $\lambda\in\boldsymbol{K}^{ac}\setminus\FF_q^{ac}$,
we have that $f(\lambda)=\sum_{i\geq 0}f_i\lambda^i$ converges in the complete field 
$\mathbb{K}^\sharp:=\widehat{\operatorname{Frac}(C\otimes_{\FF_q}\boldsymbol{K}^{ac})}$ (with $\boldsymbol{K}^{ac}$ carrying the trivial norm)
to a non-zero element.

Since $M$ is invertible, $\mathcal{E}_{w,m,\rho}$ is identically zero if and only if 
$\mathcal{E}_{w,m,\chi_t}$ is identically zero (we have supposed that $\vartheta$ has no eigenvalues in $\FF_q^{ac}$) and in the case of non-vanishing, the rank is maximal equal to $d$.
Properties 1) and 2) of our proposition hence follow from Proposition 22 of \cite{Pe}.

It remains to show the part 3); this follows from Lemma \ref{lemma2}.
Indeed, assuming that $w-1\equiv 2m\pmod{q-1}$ and $w\geq (q+1)m+1$,
let $i$ be an index between $1$ and $d$; we know that it is non-zero.
Let us assume, by contradiction, that the entries of $\mathcal{E}$ are linearly dependent; 
then, the vector space $V$ whose elements are the vectors $v\in (\mathbb{K}^\sharp)^N$ such that $\mathcal{E}(z)\cdot v=0$ for all $z\in\Omega$, is non-trivial.

Let $v$ be in $V$.
For all $\gamma\in\Gamma$, we have $$0=\mathcal{E}(\gamma(z))\cdot v=
\det(\gamma)^{-m}J_\gamma(z)^w\mathcal{E}(z)\cdot\rho(\gamma)^{-1}\cdot v.$$
This means that $\rho(\gamma)^{-1}\cdot v\in V$ so that $\rho(\gamma)(V)\subset V$
for all $\gamma\in \Gamma$. Since $\rho$ is irreducible, we thus have that $V= 
(\mathbb{K}^\sharp)^N$ but $\mathcal{E}$ is non-zero, whence a contradiction.
\end{proof}

\subsection{Vectorial Poincar\'e series associated to the representations $\rho^{\irrbis}_{\underline{t},\underline{l}}$}

We now consider the settings of \S \ref{tensorprod} and we return to independent variables $\underline{t}_s$ and to the field $\boldsymbol{K}_s=\FF_q(\underline{t}_s)$. We also consider $\underline{l}=(l_1,\ldots,l_s)$ an $s$-tuple 
with entries in $\ZZ$ which are $\geq 1$ and we note that 
the representation $\rho=\rho^{\irrbis}_{\underline{t},\underline{l}}$ of Theorem \ref{corollaryt1ts}
is normal to the depth $L=1$.
\begin{Proposition}
We have the following properties, for $w>0$ and $m\geq0$, and with $\rho$ the above considered representation.
\begin{enumerate}
\item The function $$\mathcal{E}_{w,m,\rho}:\Omega\rightarrow\operatorname{Mat}_{1\times N}(\mathbb{K}),$$ with $N=\phi_p(l_1)\cdots\phi_p(l_s)$ is well defined, holomorphic, tempered,
and satisfies
$$\mathcal{E}_{w,m,\rho}(\gamma(z))=\det(\gamma)^{-m}J_\gamma(z)^w\mathcal{E}_{w,m,\rho}(z)\cdot\rho(\gamma)^{-1},\quad \gamma\in\Gamma.$$
\item If $w':=w-l_1-\cdots-l_s\not\equiv 2m+1\pmod{q-1}$, then $\mathcal{E}_{w,m,\rho}\equiv0$.
\item With $w'$ defined as above, if $w'\equiv 2m+1$ and $w'\geq(q+1)m+1$, then $\mathcal{E}_{w,m,\rho}\neq0$.
\item If $m=0$ and $w'\equiv1\pmod{q-1}$, then $\mathcal{E}_{w,m,\rho}\neq0.$
\item In all cases in which $\mathcal{E}_{w,m,\rho}\neq0$, its entries are linearly independent over 
$\mathbb{K}$.
\end{enumerate}
\end{Proposition}
\begin{proof}
The first two properties and the last one are simple variants of the corresponding parts of Proposition \ref{proposition1}. For the third property, we consider the matrix function
$F_{\underline{l}}:\Omega\rightarrow\operatorname{Mat}_{N\times 1}(\CC_\infty)$ 
with $N=\phi_p(l_1)\cdots\phi_p(l_s)$, defined by
$$F_{\underline{l}}(z)=\operatorname{Sym}^{l_{1,0}}(F)\otimes\cdots\otimes\operatorname{Sym}^{l_{1,r_1}}(F^{(r_1)})\otimes\cdots\otimes
\operatorname{Sym}^{l_{s,0}}(F)\otimes\cdots\otimes\operatorname{Sym}^{l_{s,r_s}}(F^{(r_s)}),$$
with $F(z)=\binom{z}{1}$ and where we have used the expansions in base $p$ of $l_1,\ldots,l_s$:
$l_i=l_{i,0}+l_{i,1}p+\cdots+l_{i,r_i}p^{r_i}$ with $r_i\neq 0$ for $i=1,\ldots,s$.
Then, as in \cite{Pe}, we note that 
$$(\mathcal{E}_{w,m,\rho}\cdot F_{\underline{l}})_{t_i=\theta}=P_{w',m}$$
the Poincar\'e series of weight $w'$ and type $m$ so that we can conclude with 
\cite[Proposition 10.5.2]{Ge}. 
The property 3) is not enough to show the property 4), but we can proceed more directly by noticing
that in this case,
$$\mathcal{E}_{w,0,\rho}=\sum_{\delta\in H\backslash\Gamma}J_\gamma^{-w}\rho(\gamma)_1.$$
Hence, if we suppose that $u(z)\rightarrow\infty$, then $\mathcal{E}_{w,0,\rho}\rightarrow(0,\ldots,0,1)$.\end{proof}
The transposed of the matrix functions $\mathcal{E}_{w,m,\rho}$ are thus vectorial modular forms of weight $w$, type $m$ 
associated to the representations ${}^t\rho$ in the sense of \cite{Pe}.
\subsubsection{Eisenstein series}\label{eisensteinseries}

We consider $\FF_q$-algebra representations $\sigma_1,\ldots,\sigma_s:A\rightarrow\operatorname{Mat}_{d\times d}(\boldsymbol{K})$. Let $\sigma$ be the semi-character $\sigma_1\cdots\sigma_s$.
We set:
$$\mathcal{G}_{w,\sigma}(z)=\sideset{}{'}\sum_{(a,b)\in A^2}(az+b)^{-w}(\sigma_1(a),\sigma_1(b))\otimes\cdots\otimes(\sigma_s(a),\sigma_s(b))$$ (the dash $'$ denotes a sum which avoids the couple $(0,0)$).
This defines a holomorphic function
$$\mathcal{G}_{w,\sigma}:\Omega\rightarrow\operatorname{Mat}_{d^s\times 2d^s}(\mathbb{K}).$$ Let, on the other side, $\rho$ be the representation
$\rho:\Gamma\rightarrow\operatorname{Mat}_{2d^s\times 2d^s}(\boldsymbol{K})$ defined by 
$\rho=\rho_{\sigma_1}\otimes\cdots\otimes\rho_{\sigma_s}$. The following lemma is easy to verify.
\begin{Lemma}
We have the identity $\mathcal{G}_{w,\sigma}=L_\sigma(w)\mathcal{E}_{w,0,\rho}.$
\end{Lemma}
The matrix $L_\sigma(w)$ is the $L$-value associated to the semi-character $\sigma$ as defined in \S \ref{Lvalues}. This and Proposition \ref{Taelmanunits} suggest that the Eisenstein series $\mathcal{G}_{w,\sigma}$ could be also related to Taelman units. Of course, this is quite speculative, because at the moment, 
we do not have at our disposal any kind of metric over the spaces of vectorial modular forms that we consider,
allowing us to define an appropriate notion of unit group of Taelman in this setting. However, this seems to 
suggest the following conjecture.

\begin{Conjecture}
The $\mathbb{K}$-module of the vectorial modular forms of weight one and type $0$ associated to a representation $\rho_{\sigma_1}\otimes\cdots\otimes\rho_{\sigma_s}$
with $\sigma_1,\ldots,\sigma_s$ algebra representations $A\rightarrow\operatorname{Mat}_{d\times d}(\boldsymbol{K})$ is of rank one, generated by the Eisenstein series $\mathcal{G}_{w,\sigma}$, where
$\sigma=\sigma_1\cdots\sigma_s$.
\end{Conjecture}

\section{Other representations of $\Gamma$}\label{pathologies}

Let $\boldsymbol{K}$ be a field containing a fixed base field $k$ of positive characteristic (e.g. $k=\FF_q$) and let us consider 
a group representation $\rho:G\rightarrow\GL_N(\boldsymbol{K})$. 
The {\em essential dimension} (over $k$)
of $\rho$ is the transcendence degree over $k$ of the field generated by the entries of the image
of $\rho$. If $G$ is finite, then the essential dimension of $\rho$ is zero. In this paper, we have studied several examples in the case $k=\FF_q$ and $G=\Gamma$.
For instance, the essential dimension of the tautological representation $\Gamma\rightarrow\GL_2(A)$
is one, and the essential dimension of a representation
$\rho^{\irrbis}_{\underline{t},\underline{l}}$ as in Theorem \ref{corollaryt1ts} is $s$, the number of variables in $\underline{t}_s$.

As a conclusion of the present note, we would like to point out that there are irreducible, finite dimensional
representations of $\GL_2(k[t])$ with infinite essential dimension.
Indeed, for any field $k$, a Theorem of Nagao (see \cite[Theorem 2]{Na} and Serre, \cite[II.1.6]{Se})
asserts that
\begin{equation}\label{amalgamated}
\GL_2(k[t])\cong\GL_2(k)*_{B(k)}B(k[t]),\end{equation}
where, for a commutative ring $R$, $B(R)$ denotes the group of upper triangular matrices with 
entries in $R$ with invertible determinant and where $*_{B(k)}$ stands for the amalgamated 
product along $B(k)$. Therefore, we have the following:
\begin{Proposition}
Any automorphism $\phi$ of $\GL_2(k)$ extends to a group isomorphism between
$\GL_2(k[t])$ and the subgroup $\Phi^\infty$ of $\GL_2(k[x_1,x_2,\ldots])$ generated by 
$\GL_2(k)$ and the matrices $\sqm{\lambda}{x_i}{0}{\mu}$, where $x_1,x_2,\ldots$
are independent indeterminates over $k$ and $\lambda,\mu\in k^\times$.
\end{Proposition}
\begin{proof}
By (\ref{amalgamated}), we see that the association $\sqm{\lambda}{t^i}{0}{\mu}\mapsto\sqm{\lambda}{x_i}{0}{\mu}$ extends to give the above group isomorphism.
\end{proof}
The above proposition exhibits representations $\GL_2(k[t])\rightarrow\GL_2(\boldsymbol{K}_\infty)$
where $\boldsymbol{K}_\infty=k(x_1,x_2,\ldots)$, which have infinite essential dimension over $\FF_q$.

\begin{Remark}{\em The group $\operatorname{SL}_2(\mathbb{Z})$ has uncountably many isomorphism classes of irreducible complex representations and their explicit classification is not yet understood. A similar 
question arises with the group $\GL_2(A)$ and its representations in a complete algebraically closed field 
of characteristic $p$. 
The complete classification for $\operatorname{SL}_2(\mathbb{Z})$ is however accessible if we impose an upper bound on the dimension.
In \cite{Tu}, Tuba and Wenzl obtained a complete classification of irreducible representations of 
the braid group $B_3$ of dimension $\leq 5$ yelding a similar result for irreducible complex representations
of  $\operatorname{SL}_2(\mathbb{Z})$; it turns out that these families algebraically depend 
on finitely many parameters (eigenvalues, characters etc.). It would be nice to have a similar result
for $\GL_2(A)$.
}\end{Remark}

\subsubsection*{Acknowledgement}

The author thanks  Gebhard B\"ockle, Mihran Papikian and the Referee, for useful hints and remarks that have contributed to improve and correct the paper.

\end{document}